\documentclass[12pt]{elsarticle}
\usepackage{lineno,hyperref}
\usepackage{amsmath}
\usepackage{amsthm}
\usepackage{amssymb}
\usepackage{color}
\usepackage{comment}
\usepackage{graphicx}
\usepackage{enumitem}
\usepackage{comment}

\newtheorem{theorem}{Theorem}[section]

\newtheorem{lemma}[theorem]{Lemma}

\begin{document}

\begin{frontmatter}

\title{On distance integral and distance Laplacian integral graphs}

\author[Pirzada]{S. Pirzada}
\ead{pirzadasd@kashmiruniversity.ac.in}
\address[Pirzada]{Department of Mathematics, University of Kashmir, Srinagar, India}

\author[Pirzada]{Ummer Mushtaq}
\ead{ummermushtaq1995@gmail.com}

\author[Leoaddress1]{Leonardo de Lima}
 \ead{leonardo.delima@ufpr.br}
\address[Leoaddress1]{Departamento de Administração Geral e Aplicada, Universidade Federal do Paran\'a, Brasil}

\begin{abstract}
Let $G$ be a connected graph on $n$ vertices and let $D(G)$ and $D^{L}(G)$ be the distance and the distance Laplacian matrices associated with  $G$. A graph $G$ is said to be $D$-integral (resp. $D^L$-integral) if all eigenvalues of  $D(G)$ (resp. $D^L(G)$) are integers. In this paper, we obtain various conditions under which the graphs  $a\overline{K_m}\nabla C_n$ and $K_{p,p}\nabla C_n$ are distance integral. We also obtain conditions on $m$, $n$ under which the dumbbell graph $\boldsymbol{DB}(W_{m,n})$ is $D^L$-integral. 
\end{abstract}


\end{frontmatter}
 
\section{Introduction}
We consider a connected simple graph $G = (V, E) $, where $V$ represents the vertex set with cardinality $n$ and $E$ denotes the edge set. The adjacency matrix of graph $G$, denoted by $A(G)=a_{ij}  $, is a square matrix  with $ a_{ij} = 1$ if the vertex $v_i $ is adjacent to $ v_j$, and $a_{ij} = 0 $ otherwise. A graph is regular if each vertex has the same degree. If $M$ represents the matrix associated with a graph $G$, then $G$ is termed  $M$-integral if all eigenvalues of  $M$ are integers. The  $M$-spectrum of a graph $G$ encompasses all eigenvalues of its matrix $M$ along with their respective multiplicities. Let $d_{G}(i,j)$ be the length of the shortest path between vertices $v_i$ and $v_{j}.$ The $n \times n$ distance matrix of $G$, denoted by $D(G)$, has $d_{G}(i,j)$ as its $(i,j)$-entry. Notice that $D(G)$ is symmetric and so has real eigenvalues. Write  $Tr(G)$ for the transmission matrix of $G$ which is the
diagonal matrix of the row sums of $D(G)$. The distance Laplacian of $G$ is given by $D^{L}(G) = Tr(G)-D(G).$
The complement of $G$, denoted as $\overline{G}$, is the graph sharing the same vertex set as $ G $, where two vertices are adjacent if and only if they are not adjacent in $G $. The union of two graphs $ G_1$ and  $G_2$, denoted $G_1\cup G_2$, comprises the vertex set $V(G_1) \cup V(G_2)$ and the edge set $E(G_1) \cup E(G_2)$. The join $G_1 \nabla G_2$ of two graphs $G_1$ and $G_2 $ is the graph obtained from  $G_1\cup G_2$ by adding all possible edges from the vertices of $G_1$ to the vertices of $G_2$. Conventionally, $ K_n $ denotes the complete graph, $C_n$  represents the cycle and $ K_{p,q}$ denotes the complete bipartite graph on $p+q$ vertices. Given $m \geq 2$ and $n \geq 3,$ the generalized wheel graph $W_{m,n}$ is the graph obtained by the join $\overline{K_{m}} \nabla C_{n}$. The dumbbell graph, denoted by $DB(W_{m,n})$, is obtained from two copies of generalized wheel graph $W_{m,n}$ by connecting
$m$ vertices in one copy with the corresponding vertices in the other copy  \cite{19b}. For more definitions and notations, we refer to \cite{sp}.\\
\indent In spectral graph theory, a fundamental question is about identifying graphs where all the eigenvalues of a matrix associated with the graph are integers. This quest gained momentum in 1974 when Harary and Schwenk \cite{5} introduced the concept of integral graphs, igniting widespread interest and prompting extensive research in this area. In 2002, Balaliska \emph{et al.} \cite{2} presented a comprehensive survey of findings concerning integral graphs. For further insights and results, we refer to \cite{1,MH14,2,3,9,13}.\\
\indent Regarding the $D$-integral graph, the computational results of \cite{12} point out that there are only 49 $D$-integral graphs up to 9 vertices. The rarity of $D$-integral graphs is a motivation to characterize graphs with this property, and some papers in the literature have addressed this subject. In \cite{4}, Ili\'{c} characterized the distance spectra of integral circulant graphs and demonstrated their $D$-integrality. In 2015, Pokorn\'{y} \textit{\emph{et al.}}\cite{12} delineated $D$-integral graphs within various classes, including complete split graphs, multiple complete split-like graphs, extended complete split-like graphs, and multiple extended complete split-like graphs, also proving the non-$D$-integrality of nontrivial trees. In the same year, Yang and Wang\cite{15} established a necessary and sufficient condition for the $D$-integrality of complete $r$-partite graphs $K_{p_1,p_2,\dots,p_r} \cong K_{a_{1}p_{1},a_{2}p_{2},\dots,a_{s}p_{s}}$ and generated numerous new classes of $D$-integral graphs for $s=1, 2, 3, 4$. In 2016, Zhao \emph{et al.} \cite{11} provided necessary and sufficient conditions for determining the $D^Q$-integrality of complete $r$-partite graph such as $K_{{p}_1,p_2,\dots,p_r}\cong K_{a_{1}p_{1},a_{2}p_{2},\dots,a_{s}p_{s}}$ and introduced new classes of $D^Q$-integral graphs for $s = 1, 2, 3$. Subsequently, in 2017, da Silva \emph{et al.}\cite{P14} investigated the concept of $D^L$-integrality and integrality of the distance signless Laplacian eigenvalues in various types of graphs, including complete split graphs, multiple complete split-like graphs, extended complete split-like graphs, and multiple extended complete split-like graphs, drawing on Pokorn\'{y}'s work \cite{12}. Additional insights and advancements in this field can be found in \cite{7,10}. In 2023, Lu \emph{et al.} \cite{9} obtained all $D$-integral generalized wheel graphs $aK_m\nabla C_n$. \\

\indent The rest of this paper is organized as follows. In Section 2, we completely characterize all $D$-integral graphs of the forms   $a\overline{K_m}\nabla C_n$ and $K_{p,p}\nabla C_n,$ and prove that there are no $D$-integral dumbbell graphs. In Section 3, we derive all $D^L$-integral dumbbell graphs $\boldsymbol{DB}(W_{m,n})$.

\section{$D$-integral of $a\overline{K_m}\nabla C_n$ and $K_{p,p}\nabla C_n$ }

In this section, we shall provide the distance spectrum of the generalized graph $a\overline{K_m}\nabla C_n$ and $K_{p,p}\nabla C_n$. We denote $a\overline{K_m}\nabla C_n$ by $EGW(a, m, n)$ and completely determine all $D$-integral  $EGW(a, m, n)$ graphs.
The following observations will be used in the sequel.
\begin{lemma}\label{l1}\em \cite{16}
	The adjacency spectrum of the cycle $C_n$ is
		\begin{align*}
		\Bigg\{2cos \left(\frac{2k\pi}{n}\right): k=1,2,\dots,n-1\Bigg\}.
	\end{align*}
\end{lemma}
\begin{lemma}\label{l2}\em\cite{16}
	Let $M$ be a square matrix of order $n$ that can be written in the blocks as
		\[ M=\begin{pmatrix}
			M_{11}& M_{12}  &\dots&M_{1k}\\
			M_{21}&M_{22}&\dots&M_{2k}\\
			\dots&\dots&\dots&\dots\\
			\dots&\dots&\dots&\dots\\
			M_{k1}&M_{k2}&\dots&M_{kk}
	\end{pmatrix}\]
	where $	M_{ij},1\leq i,j\leq k,$ is the $n_i\times m_j$ matrix such that its lines have constant row sum equal to $c_{ij}$. Let  $\overline{M}=
	[c_{ij}]_{k\times k}$. Then the eigenvalues of $\overline{M}$ are also eigenvalues of $M$.
	\begin{lemma}\label{l3}\em\cite{17}
	Let $M$ be defined as in Lemma \ref{l2} such that $M_{ij}= s_{ij}J_{n_{i},n_{j}}$ for $i\neq j$, and $M_{ii} =s_{ii}J_{n_{i},n_{i}}+ p_iI_{n_{i}}$. Then the quotient matrix of $M$ is $B=b_{ij}$ with $b_{ij}=s_{ij}n_j$ if $i\neq j$, and $b_{ii}=s_{ii}n_i+p_i$. Moreover,
	\begin{align*}
		\sigma(M)=\sigma(B)\cup\Big\{p_1^{n_1-1},p_2^{n_2-1},\dots,p_k^{n_k-1}\Big\},
	\end{align*}
		where $\sigma(A)$ represents the spectrum of matrix $A$.
	\end{lemma}
\end{lemma}
The following lemma gives the distance spectrum of $\overline{K_m}\nabla C_n$.
\begin{lemma}\label{l4}
	The distance spectrum of $\overline{K_m}\nabla C_n$ is
	\begin{align*}
	\Bigg\{-2^{m-1},(m+n-3) \pm\sqrt{(m+n-3)^2-(3mn-8m-4n+8)},\\
	-2-2cos \left(\frac{2k\pi}{n}\right): k=1,2,\dots,n-1\Bigg\}.
		\end{align*}
	\end{lemma}
\begin{proof}
		Consider $W_{m,n}=\overline{K_m}\nabla C_n, m\geq1, n\geq3$, the join of the empty graph with $m$ vertices and 2-regular cycle graph with $n$ vertices. The distance matrix of $W_{m,n}$ can be written as
		\begin{align*}
			D(W_{m,n})=\begin{bmatrix}
				A&B\\
				C & D
			\end{bmatrix},
		\end{align*}
where $A=2(J-I)_{m\times m}$, $B=J_{m\times n}$, $C= J_{n\times m}$, $D=2(J-I)_{n\times n}-A(C_n)$,
and $J_{m,n}$ is all ones matrix, $A(C_n)$ is the adjacency matrix of the cycle $C_n$.\\		
\indent Assume that $V(\overline { K_m})=\lbrace v_1, v_2,\dots, v_m\rbrace$ and $V(C_n)=\{u_1,u_2,\dots,u_n\}$ be the vertex set of the graphs $\overline{ K_m}$ and  $C_n$, respectively. Then the vertex set of  $W_{m,n}$ is $V(W_{m,n})=V(\overline { K_m})\cup V(C_n)$.\\
\indent Since $\overline{ K_m}, m\geq 2$ is an empty graph of order $m$, in $D(W_{m,n})$, the shortest distance between the vertices $v_i\in \overline{ K_m},~ i= 1,2,\dots,m$ is two. Let $e_m$ be all ones vector of length $m$, say $e_m=(1,1,\dots,1)^T$. In $D(W_{m,n})$, the block matrix $A$ has eigenvector $[e_m\hspace{5mm} 0^T_n]_{(m+n)\times1}$ corresponding to the eigenvalue $2m-2$, while the remaining eigenvectors corresponding to the eigenvalue $-2$ with multiplicity $m-1$ are orthogonal to $[e_m\hspace{5mm} 0^T_n]_{(m+n)\times1}$.\\
Also, $A(C_n)$ has all ones vector $e_n=(1,1,\dots,1)^T$ as an eigenvector corresponding to the eigenvalue $2$, the remaining eigenvectors are orthogonal to $e_n$. In $D(W_{m,n})$, the block matrix $D$ has an eigenvector  $[0^T_m\hspace{5mm} e_n]_{(m+n)\times1}$ corresponding to the eigenvalue $2n-4$ while the remaining eigenvectors corresponding to the eigenvalues $-2-2cos  \left(\frac{2k\pi}{n}\right)$, $k= 1,2,\dots,n-1$, are orthogonal to  $[0^T_m\hspace{5mm} e_n]_{(m+n)\times1}$. The remaining two distance eigenvalues are $\lbrace (m+n-3)\pm\sqrt{(m+n-3)^2-(3mn-8m-4n+8)}\rbrace$, which are the zeros of the characteristic polynomial of the quotient matrix $\begin{bmatrix}
	2m-2 & n\\
	m &2n-4
\end{bmatrix}$. This completes the proof.
	\end{proof}

	In the following lemma, we find the distance spectrum of the join of $a$ copies of the empty graph  $\overline{K_m}$ and the cycle $C_n$.

	\begin{lemma}\label{l5}
		The distance spectrum of the extended $EGW(a,m,n)$ consists of the eigenvalues
		\begin{align*}
			\Bigg\{-2^{(am-1)},(am+n-3)\pm\sqrt{(am+n-3)^2-(3amn-8am-4n+8)},\\	-2-2cos \left(\frac{2k\pi}{n}\right); k=1,2,\dots,n-1\Bigg\}.
		\end{align*}
	\end{lemma}
	\begin{proof}
			Consider the graph $a\overline{K_m}\nabla C_n,~ a\geq 2, m\geq1, n\geq3$, the join of $a$ copies of the empty graph  $\overline{K_m}$ with $m$ vertices and 2-regular cycle graph with $n$ vertices. The distance matrix of $a\overline{K_m}\nabla C_n$ can be written as
			
			\[ D(a\overline{K_m}\nabla C_n)=\begin{pmatrix}\begin{array}{cccc|c}
					X_{11}& X_{12}  &\dots&X_{1a}&J_{a\times n}\\
					X_{21}&X_{22}&\dots&X_{2a}&J_{a\times n}\\
					\dots&\dots&\dots&\dots&\dots\\
					\dots&\dots&\dots&\dots&\dots\\
					X_{a1}&X_{a2}&\dots&X_{aa}&J_{a\times n}\\ \hline
					J_{n\times a}&J_{n\times a}&\dots&J_{n\times a}&2(J-I)_{n\times n}-A(C_n),\\
				\end{array}\\
			\end{pmatrix}\]
where \begin{align*}
		X_{ij}=\begin{cases}
			 2(J-I)_{a\times a}, \qquad \mbox{if} ~ i=j\\
		2J_{a\times a},\qquad \mbox{if}~  i\neq j,
		\end{cases}
		\end{align*}
	 $J$ is the all ones matrix and $A(C_n)$ is the adjacency matrix of cycle $C_n$.\\
\indent	Let $V_{p}(\overline{K_m})=\{u^p_1,u^p_2,\dots,u^p_m\}$ and $V(C_n)=\{ v_1,v_2,\dots,v_n\}$ be the vertex sets of the graphs $p^{th}$ copy of $\overline{K_m}$, where $1\leq p\leq a$, and $C_n$. The vertex set of $a\overline{K_m}\nabla C_n$ is $V(a\overline {K_m}\nabla C_n)=\cup_{p=1}^{a}V_{p}(\overline{K_m})\cup V(C_n)$.\\
\indent	Let $v^1_m=(1,1,\dots,1)^T$ be all ones vector of length, say $m$. Then in the matrix $ D(a\overline{K_m}\nabla C_n)$,  $v^1_m$ is an eigenvector of the block $X_{11}$ corresponding to the eigenvalue $2m-2$, while the remaining eigenvalues corresponding to $-2$ with algebraic multiplicity $m-1$ are orthogonal to $v^1_m$. Under these circumstances, $-2$ is an eigenvalue of $ D(a\overline{K_m}\nabla C_n)$ of multiplicity $m-1$ and the corresponding eigenvectors are orthogonal to $[v^1_m\hspace{3mm}, 0^T_n,\dots, 0^T_m, 0^T_n]_{(am+n)\times1}$. Similarly, for each diagonal block matrix $X_{ii}, ~1\leq i\leq a$, $-2$ is an eigenvalue of $ D(a\overline{K_m}\nabla C_n)$ of multiplicity $m-1$ and the corresponding eigenvectors are orthogonal to \begin{align*}
	 [0^T_m,0^T_m,\dots,v^i_m\hspace{3mm}\dots, 0^T_m, 0^T_n]_{(am+n)\times1}.
	\end{align*}
	Also, $A(C_n)$ and $2(J-I)_{n\times n}$ have all ones vector $e_n=(1,1,\dots,1)^T$ as an eigenvector corresponding to the eigenvalue $2$ and $2n-2$, respectively. Then, the block $2(J-I)_{n\times n}-A(C_n)$ has an eigenvector $e_n$ corresponding to the eigenvalue $2n-4.$ Therefore, $-2-2cos (\frac{2k\pi}{n}), k=1,2,\dots,n-1$ are the eigenvalues of $ D(a\overline{K_m}\nabla C_n)$ and the corresponding vectors are orthogonal to $[0^T_m,0^T_m,\dots, 0^T_m,e_n]_{(am+n)\times1}$.\\
\indent	In this way, we get  $-2-2cos (\frac{2k\pi}{n}), k=1,2,\dots,n-1$ and $-2$  with multiplicity $am-a$ as eigenvalues of $ D(a\overline{K_m}\nabla C_n)$. The remaining eigenvalues are the eigenvalues of the quotient matrix
	\[\begin{pmatrix}\begin{array}{cccc|c}
		2m-2&2m&\dots&2m&n\\
		\dots&\dots&\dots&\dots&\dots\\
	        2m&2m&\dots&2m-2&n\\	\hline
		n&n&n&n&2n-4
			\end{array}\\
	\end{pmatrix}_{(a+1)\times(a+1)}.\]
	By Lemma \ref{l3}, clearly $-2$ is an eigenvalue of the above matrix with multiplicity $a-1$ and the  remaining two eigenvalues are the zeros of the characteristic polynomial of the following matrix
	\begin{align*}
		\begin{bmatrix}
			2am-2&n\\
			am&2n-4
		\end{bmatrix}.
	\end{align*}
	By combining all the above, the complete spectrum of $ D(a\overline{K_m}\nabla C_n)$ is
	\begin{align*}
			\Bigg\{-2^{(am-1)},(am+n-3)\pm\sqrt{(am+n-3)^2-(3amn-8am-4n+8)},\\	-2-2cos \left(\frac{2k\pi}{n}\right): k=1,2,\dots,n-1\Bigg\}.
	\end{align*}
		\end{proof}

Now, we have the following result.

\begin{theorem}
The graph $\overline{K_m}\nabla C_n$ with $m+n$ vertices is $D$-integral if and only if one of the following holds:\\
(i)  $n=3,~ m=1$, (ii)  $n=3, ~  m= 4$, (iii)  $n=4,~   m= 2$, (iv)  $n=6, ~  m= 4$, (v)  $n=6, ~  m= 12$.
\end{theorem}			
\begin{proof} By Lemma \ref{l4}, the distance spectrum of $\overline{K_m}\nabla C_n$ consists of the eigenvalue $-2$ with multiplicity $m-1$ and $-2-2cos (\frac{2k\pi}{n}), k=1,2,\dots,n-1$ and the remaining two eigenvalues are $$(m+n-3)\pm\sqrt{(m+n-3)^2-(3mn-8m-4n+8)}.$$
\noindent Clearly, $(m+n-3)\pm\sqrt{(m+n-3)^2-(3mn-8m-4n+8)}$ are integers if and only if $(m+n-3)^2-(3mn-8m-4n+8)$ is a perfect square. Also, it is apparent that $2cos (\frac{2k\pi}{n})$ is an integer for any $1\leq k\leq n-1$ if and only if $n\in \{3,4,6\}$. Therefore, $\overline{K_m}\nabla C_n$ is $D$-integral if and only if $(m+n-3)^2-(m-4n+8)$ is a perfect square and $n\in \{3,4,6\}$.\\
\indent	 To	establish sufficiency involves straightforward calculations, indicating that $\overline{K_m}\nabla C_n$ becomes $D$-integral when condition (i), (ii), (iii) or (iv) holds. Now, delve  into the necessity aspect. Let $t=(m+n-3)^2-(3mn-8m-4n+8)=c^2$. We consider the following three cases.\\
	\textbf{	Case 1.} Let $n=3$. Then, we have  $t=m^2-m+4=c^2$, that is, $m^2-m+4-c^2=0$. This gives, $m=\frac {1\pm  \sqrt{4c^2-15}}{2}$. Since $m$ is an integer which is possible only if $4c^2-15$ is an odd perfect square number. Let
$4c^2-15=p^2$. This implies that $4c^2-p^2=15$, further implies that $(2c-p)(2c+p)=15$, implies that $c=4$ or  $c=2$.
Substituting back values of $c$, we get $m=4$ or $m=1$.\\
		\textbf{	Case 2.} Let $n=4$. We have  $t=(m+1)^2-4m+8=c^2$, that is, $ 8=c^2-(m-1)^2$. This gives, $8=(c+1-m)(c-1+m)$. Now, $8=8\times1=2\times4$. For the case $(c+1-m)(c-1+m)=8\times1$, we have $2c=9$, which is not possible. For the case $(c+1-m)(c-1+m)=2\times4$, we have $c=3$, which gives $m=2$.\\
				\textbf{	Case 3.} Let $n=6$. So, we have  $t=(m+3)^2-10m+16=c^2$, that is, $21=c^2-(m-2)^2$. We have $21=3\times7=1\times21$. Now, For the case $(c-m+2)(c+m-2)=1\time211$, we have $c=11$, which gives $m=12$. For the case $(c-m+2)(c+m-2)=3\times7$, we have $c=5$, which gives $m=4$.
	\end{proof}

\begin{theorem}
	Let $a\geq1$ and $n\geq3$. Then the extended graph $EGW(a,m,n)$ on $am+n$ vertices  is $D$-integral if and only if the ordered triple $(a,m,n)\in$ S, where
		\begin{align*}
		S=\{(1,4,3),(4,1,3),(2,2,3),(1,1,3), (1,2,4),(2,1,4),(1,4,6),(2,2,6),(4,1,6),\\(1,12,6),(2,6,6),(3,4,6),(4,3,6),(6,2,6),(12,1,6)\}
		\end{align*}
\end{theorem}	
\begin{proof} By Lemmas \ref{l4} and \ref{l5}, the distance spectrum of $a\overline{K_m}\nabla C_n$ consists of the eigenvalues $-2$ with multiplicity $am-1$ and $	-2-2cos (\frac{2k\pi}{n}), k=1,2,\dots,n-1$, and the remaining two eigenvalues are $$(am+n-3)\pm\sqrt{(am+n-3)^2-(3amn-8am-4n+8)}.$$
\noindent	It is clear that $(m+n-3)\pm\sqrt{(am+n-3)^2-(3amn-8am-4n+8)}$ are integers if and only if $(am+n-3)^2-(3amn-8am-4n+8)$ is a perfect square. Moreover, it is obvious that $2cos (\frac{2k\pi}{n})$ is an integer for any $1\leq k\leq n-1$ if and only if $n\in \{3,4,6\}$. Therefore, the extended graph $EGW(a,m,n)$ is $D$-integral if and only if $(am+n-3)^2-(3amn-8am-4n+8)$ is a perfect square and $n\in \{3,4,6\}$.\\
 Let $t=(am+n-3)^2-(3amn-8am-4n+8)=c^2$. We consider the following three cases.\\
\textbf{	Case 1.} Let $n=3$. Then, we have  $t=(am)^2-am+4=c^2$, that is, $(am)^2-am+4-c^2=0$. This gives, $am=\frac {1\pm  \sqrt{4c^2-15}}{2}$. Since $am$ is an integer which is possible only if $4c^2-15$ is an odd perfect square. We suppose that $4c^2-15=p^2$, for some positive integer $p$. Then $4c^2-p^2=15$, that is, $(2c-p)(2c+p)=15$. Since $15=1\times15=5\times 3$, this gives $c=4$ or $c=2$. By substituting the value of $c$ in $am=\frac {1\pm  \sqrt{4c^2-15}}{2}$, we get $am=4$ or $am=1$. Since $4=4\times1=2\times 2=1\times4$, therefore the possible values of $a$ and $m$ are $a=1,~ m=4$ or $a=4,~ m=1$ or $a=2,~m=2$. Similarly, $am=1$ gives the possible values as $a=1, ~m=1$. Thus, for $n=3$, the possible extended distance integral graphs $EGW(a,m,n)$ are $EGW(1,4,3), ~EGW(4,1,3), ~EGW(2,2,3), ~EGW(1,1,3)$.\\
	\textbf{	Case 2.} Let $n=4$. We have  $t=(am+1)^2-4am+8=c^2$, that is, $(am-1)^2+8=c^2$. This gives $ c^2- (am-1)^2=8$. Now by factorizing, we get  $(c-am+1)(c+am-1)=2\times 4=1\times8$. Then the possible values of $c$ is 3. By back substituting the value of $c=3$, we get $am=2$.\\
\indent	Since $2=1\times 2=1\times2$, therefore, the possible values of $a$ and $m$ corresponding to $n=4$ are $a=1, ~m=2$  or $a=1,~ m=1$.  Therefore, for $n=4$, the possible extended distance integral graphs are $EGW(1,2,4)$ and $ EGW(2,1,4)$.\\
	\textbf{Case 3.} Let $n=6$. In this case, we have  $t=(am+3)^2-10am+16=c^2$, that is, $(am-2)^2+21=c^2$. This gives $(c-am+2)(c+am-2)=1\times21=3\times7$. Here the possible values $c$ are $5 $ and $11$. By back substituting the values of $c$, we get $am=4$ and $am=12$.  \\
	Since $12=1\times 12=2\times6=3\times4=4\times3=6\times2=12\times 1$ and $4=1\times 4=2\times2=4\times 1$, therefore, the possible values of $a$ and $m$ corresponding to $n=6$ are $a=1, ~m=4$  or $a=4,~ m=1$ or $a=2, ~m2$ or $a=1, ~m=12$ or $a=2, ~m=6$ or $a=3,~ m=4$ or $a=4,~ m=3$ or $a=6,~m=2$ or $a=12,~m=1$ .  Thus, for $n=6$, the possible extended  distance integral graphs $EGW(a,m,n)$  are
$EGW(1,4,6), EGW(2,2,6), EGW(4,1,6), EGW(1,12,6), EGW(2,6,6),\\ EGW(3,4,6)$, $EGW(4,3,6)$, $EGW(6,2,6)$, $EGW(12,1,6)$.
\end{proof}

The following observation gives the distance spectrum of the graph $K_{p,p}\nabla C_n$.

\begin{theorem}\label{t3}
	The distance spectrum of $K_{p,p}\nabla C_n$ is
	\begin{align*}
	\Bigg\{-2^{2p-2},-2-2cos\left(\frac{2k\pi}{n}\right); k=1,2,\dots,n-1,\\ \frac{2n+3p-6\pm\sqrt{(2n+3p-6)^2-16pn+48p+16n-32}}{2}\Bigg\}
	\end{align*}
	\end{theorem}
\noindent \textbf{Fact 1.}  An ordered triplet of positive integers $(a,b,c)$ is called Pythagorean triplet if $a^2+b^2=c^2$. We note that $(3,4,5)$ is the only Pythagorean triplet whose one component is $4$.

The following result gives a necessary and sufficient condition for the graph $K_{p,p}\nabla C_n$ to be $D$-integral.
\begin{theorem}
The graph $K_{p,p}\nabla C_n$ is $D$-integral if and only if (i)  $n=3,\hspace{1mm} p=1$ and (ii)  $n=6, \hspace{1mm}  p= 4$.
\end{theorem}
\begin{proof} From Theorem \ref{t3}, the distance spectrum of the graph $K_{p,p} \nabla C_n $ includes eigenvalues of $-2$ of multiplicity $2p-2$, as well as $-2-2cos(\frac{2k\pi}{n})$, where $k$ ranges from $1$ to $n-1$. Additionally, the remaining two eigenvalues are
\begin{align*}
					\frac{2n+3p-6\pm\sqrt{(2n+3p-6)^2-16pn+48p+16n-32}}{2}.
\end{align*}
It is evident that $2n+3p-6\pm\sqrt{(2n+3p-6)^2-16pn+48p+16n-32}$ will yield integer values if and only if $(2n+3p-6)^2-16pn+48p+16n-32$ is a perfect square. Furthermore, it is clear that $2\cos(\frac{2k\pi}{n})$ results in the integer values for any $1\leq k\leq n-1$ if and only if $n$ belongs to the set $\{3,4,6\}$.
Consequently, the graph $K_{p,p}\nabla C_n$ is $D$-integral if and only if $(2n+3p-6)^2-16pn+48p+16n-32$ is a perfect square and $n$ belongs to the set $\{3,4,6\}$. Let $t=(2n+3p-6)^2-16pn+48p+16n-32=c^2$. Now, we discuss the following three cases of $n$.\\
\textbf{Case 1.} Let $n=3$. Then, we have $t=(3p)^2+4^2=c^2$. Since each term on both sides is a perfect square and the left-hand side is the sum of two perfect square numbers with one number $4^2$, so according to Fact 1, $3p$ should be equal to $3$, implying that $p$ equals to $1$. Therefore, the only possible distance integral graph in this case is  $K_{1,1}\nabla C_3$.\\
\textbf{Case 2.} Let $n=4$. In this case, we have $t=(3p+2)^2-16p+32=c^2$, that is, $9p^2-4p+36=c^2$. This gives $p=\frac{4\pm\sqrt{4c^2-1040}}{18}$. We first check the values of $c$, for which $4c^2-1040$ is a perfect square number. Assume that $4c^2-1040=q^2$. Then $(2c-q)(2c+q)=1040$. The different possibilities to write $1040$ as a product of two factors are
				\begin{align*}
					1040=1\times1040=2\times520=4\times260=8\times130=10\times104\\
					=20\times52=40\times26=16\times65=80\times13=5\times208.
				\end{align*}
We eliminate the possibilities $1040=1\times1040=2\times520=8\times130=40\times26=16\times65=80\times13=10\times104=5\times208$, because none of these pairs has the sum a multiple of $4$. So we only consider two cases, namely $	 (2c-q)(2c+q)=1040=4\times260$ and $(2c-q)(2c+q)=1040=20\times52$. For the case $(2c-q)(2c+q)=4\times260$, we get $4c=264$, implying $c=66$ and for the case $(2c-q)(2c+q)=20\times52$, we have $c=18$. By substituting these values of $c$ in $p=\frac{4\pm\sqrt{4c^2-1040}}{18}$, we get $p=\frac{20}{18}$ and $p=\frac{132}{18}$, which is absurd. Thus, for $n=4$, there is no value of $p$ for which  $K_{p,p}\nabla C_n$ is distance integral.\\
\textbf{	Case 3.} Let $n=6$. Here, we have $t=(3p+6)^2-48p+64=c^2$, that is, $9p^2-12p+100=c^2$. This gives $p=\frac{12\pm\sqrt{4c^2-3456}}{18}$. Continuing with the same approach as in Case 2, we find the only integer value of $p$ is $4$. Therefore, the only possible $D$-integral graph in this case is $K_{4,4}\nabla C_6$. This completes the proof.
\end{proof}

Next, we show that there is no dumbbell graph that is D-integral.

\begin{lemma}\cite{19}\label{121}
		The distance spectrum of dumbbell graph $\boldsymbol{DB}(W_{m,n})$ consists of the set of eigenvalues
			\begin{align*}
			\Bigg\{-4^{m-1},0^{m-1},\frac{1}{2}[-(m+n+4) \pm\sqrt{(m+n+4)^2-16m}],\\\frac{1}{2}[(5m+5n-8) \pm\sqrt{25m^2+25n^2-14mn}],
			-2-2cos \Bigg(\frac{2k\pi}{n}\Bigg); k=2,\dots,n\Bigg\},
		\end{align*}
			where each of the eigenvalues $-2-2cos (\frac{2k\pi}{n}); k=2,3,\dots,n$ is of multiplicity $2$.\\
	\end{lemma}
	\begin{theorem}
		There are no $D$-integral  dumbbell graphs.
	\end{theorem}
	\noindent \textbf{Proof.} By Lemma \ref{121} , the $D$-spectrum of the dumbbell graph $\boldsymbol{DB}(W_{m,n})$ is given by
		\begin{align*}
		\Bigg\{-4^{m-1},0^{m-1},\frac{1}{2}[-(m+n+4) \pm\sqrt{(m+n+4)^2-16m}],\\\frac{1}{2}[(5m+5n-8) \pm\sqrt{25m^2+25n^2-14mn}],
		-2-2cos (\frac{2k\pi}{n}); k=2,\dots,n\Bigg\},
		\end{align*}
		where each of the eigenvalues $-2-2cos (\frac{2k\pi}{n}); k=2,3,\dots,n$ is of multiplicity $2$.\\
		Since $-(m+n+4)$ and $(m+n+4)^2-16m$ are of the same parity and also  $(5m+5n-8)$ and $25m^2+25n^2-14mn $ are of the same parity, therefore it suffices to find the values of $m$ and $n$ for which both $\frac{1}{2}[-(m+n+4) \pm\sqrt{(m+n+4)^2-16m}$ and $\frac{1}{2}[(5m+5n-8) \pm\sqrt{25m^2+25n^2-14mn}]$ are perfect square numbers. As $cos \frac{2k\pi}{n}$ is an integer only if $n\in \{3,4,6\}$, therefore, $DB(W_{m.n})$ is $D$-integral if and only if $\sqrt {(m+n+4)^2-16m}$ and $\sqrt{25m^2+25n^2-14mn}$ are integers, where $n\in \{3,4,6\}$. Assume that $t=(m+n+4)^2-16m=c^2$. We consider the following three cases.\\
		\textbf{Case 1}.
	Let $n=3$. We have  $t=(m+7)^2-16m=c^2$, that is, $48 =c^2-(m-1)^2$ which is equivalent to $48=(c-(m-1))(c+(m+1))$. As $48=1\times48=2\times 24=3\times16=6\times8=4\times12$, the sum of factors of $48$ should be the multiple of $2$. Therefore, we consider only $(c-(m-1))(c+(m-1))=4\times12$, $(c-(m-1))(c+(m-1))=2\times24$ and $(c-(m-1))(c+(m-1))=6\times8$. This gives $c=8$, $c=13$ and $c=7$. Substituting back the values of $c$, we get $m=2$, $ m=5$  and $m=12$ as the only integer values of $m$. Now, $\sqrt{25m^2+25n^2-14mn}$ is not a perfect square number when $m=2,5,12$ and $n=3$. Therefore, $\boldsymbol{DB}(W_{m,n})$ is not $D$-integral in this case.\\
		\textbf{Case 2}. \\
		Let $n=4$. Then $t=(m+8)^{2}-16m=c^2$, that is, $64=(c-m)(c+m)$. So the possible values of $c$ are $8,10$ and $17$. Corresponding to these values, the possible values of m are $15$ and $6$. But for $m=6,15$ and $n=4$, we note that  $\sqrt{25m^2+25n^2-14mn}$ is not a perfect square number. Thus, $\boldsymbol{DB}(W_{m,n})$ is not $D$-integral in this case.\\
		\textbf{Case 3}.\\
		For $n=6$, we see that $t=(m+2)^{2}+96=c^2$. This gives, $(c-(m+2))(c+(m+2))=96$. Proceeding similarly as in above cases, we observe that there exists no positive integral value of $m$, for which both $\sqrt {(m+n+4)^2-16m}$ and $\sqrt{25m^2+25n^2-14mn}$ are integers. Hence no distance integral dumbbell graphs exist in this case.\qed

\section{$D^L$-integral dumbbell graphs}

In this section, we find all the possible distance Laplacian integral dumbbell graphs. We prove that there are only $8$ scattered graphs in this family. The following result gives the distance Laplacian spectrum of a dumbbell graph.

\begin{lemma}\label{l6}\em\cite{7}
		The distance Laplacian spectrum of dumbbell graph $\boldsymbol{DB}(W_{m,n})$ consists of the eigenvalues
	\begin{align*}
	\Bigg\{0,3m+3n,(5m+3n)^{m-1},(5m+3n-4)^{m-1},\Big\{(3m+5n-2)^{n-1}+2cos\left(\frac{2k\pi}{n}\right)\Big\}^{2};\\ k=1,2,\dots,n-1, \frac{1}{2} \Big[(9m+9n-4)\pm \sqrt {(3m-3n-4)^{2}+4mn}\Big]\Bigg\}
	\end{align*}
	\end{lemma}

 Now, we show that there exist only $9$ distance Laplacian integral dumbbell graphs.																 
\begin{theorem}
	The only $D^L$-integral  dumbbell graphs are $\boldsymbol{DB}(W_{4,3})$, $\boldsymbol{DB}(W_{5,3})$,\\ $\boldsymbol{DB}(W_{5,4}),\boldsymbol{DB}(W_{6,4})$, $\boldsymbol{DB}(W_{14,4})$,$\boldsymbol{DB}(W_{7,6})$, $\boldsymbol{DB}(W_{8,6})$, $\boldsymbol{DB}(W_{12,6})$ and $\boldsymbol{DB}(W_{19,6})$.
\end{theorem}
\begin{proof} The $D^L$-spectrum of the dumbbell graph $\boldsymbol{DB}(W_{m,n})$ is given in Lemma \ref{l6}. Since $(9m+9n-4)$ and $(3m-3n-4)^2 $ are of the same parity, therefore, $\frac{1}{2} \Big[(9m+9n-4)\pm \sqrt {(3m-3n-4)^{2}+4mn}\Big]$ is an integer if and only if $(3m-3n-4)^{2}+4mn$ is a perfect square. Also, $cos \frac{2k\pi}{n}$ is an integer only if $n\in \{3,4,6\}$. Therefore, $DB(W_{m.n})$ is $D^L$-integral if and only if $\sqrt {(3m-3n-4)^{2}+4mn}$ is an integer where $n\in \{3,4,6\}$. Assume that $t=(3m-3n-4)^{2}+4mn=c^2$. We consider the following three cases.\\
			\textbf{Case 1}. Let $n=3$. In this case, we have  $t=(3m-11)^{2}+48=c^2$, that is, $48 =c^2-(3m-11)^2$, which is equivalent to $48=(c-(3m-11))(c+(3m-11))$. Since $48=1\times48=2\times 24=3\times16=6\times8=4\times12$, so the sum of the factors of $48$ should be the multiple of $2$. Therefore, we consider only $(c-(3m-11))(c+(3m-11))=4\times12$, $(c-(3m-11))(c+(3m-11))=2\times24$ and $(c-(3m-11))(c+(3m-11))=6\times8$. This gives $c=8$, $c=13$ and $c=7$. Substituting back the values of $c$, we get $m=4$ and $ m=5$ to be the only integer values of $m$. Thus, the only distance Laplacian integral dumbbell graphs when $n=3$, are  $DB(W_{4,3})$ and  $DB(W_{5,3})$.\\
	\textbf{Case 2}. Let $n=4$. Here, we have  $t=(3m-16)^{2}+16m=c^2$, that is, $9m^2-80m+256-c^2=0$. This gives, $m=\frac {80 \pm  \sqrt{36c^2-2816}}{18}$. Since $m$ is an integer, so $36c^2-2816$ should be perfect square number. Let	 $36c^2-2816=p^2$, that is, $(6c-p)(6c+p)=2816$. Now the different possibilities for the factors  $(6c-p)$ and $(6c+p)$ are
		 \begin{align*}
			2816=4\times704=2\times1408=1\times2816=256\times11=128\times22\\
			2816=64\times44=32\times88=16\times176=8\times352.
		\end{align*}
		The cases $(6c-p)(6c+p)=1\times2816$, $(6c-p)(6c+p)=256\times46$,  $(6c-p)(6c+p)=128\times22$ and $(6c-p)(6c+p)=2\times1408$ are not possible, because the sum of these factors is not a multiple of $6$. For the case $(6c-p)(6c+p)=8\times52$, $c=30$ and the only possible positive integral value of $m$ corresponding to $c=30$ is $m=14$. Similarly, for the cases $(6c-p)(6c+p)=16\times 176,(6c-p)(6c+p)=32\times 8,~ (6c-p)(6c+p)=64\times 44 $ and $(6c-p)(6c+p)=4\times704$, the only possible integer values of $m$ are $5$ and $6$. Therefore, the only distance Laplacian integral dumbbell graphs when $n=4$ are  $DB(W_{5,4}),DB(W_{6,4})$ and  $DB(W_{14,4})$.\\
		\textbf{ Case 3}. Taking $n=6$, we see that  $t=(3m-22)^{2}+24m=c^2$, that is, $(3m-18)^{2}+160=c^2$. This gives $c^2-(3m-18)^{2}=160$, which implies that $(c-(3m-18){2})(c-(3m-18){2})=160$. Now the different possibilities to write $160$ as product of two factors are
		\begin{align*}
			160=1\times160=5\times32=10\times 16=20\times8=40\times4=80\times2.
		\end{align*}
		The first two cases $160=1\times160=5\times32$ and $1605=5\times32$ are not possible, as the sum of these factors is not a multiple of $2$. For the case $160=10\times 16$,  $c=13$ and the only possible positive integral value of $m$ corresponding to $c=13$ is $m=7$. For the case $(c-(3m-18){2})(c-(3m-18){2})=20\times8 $, we have $c=14$. Using back the value of $c$ gives $m=8$. Similarly for the remaining two cases, the values of $m$ are $12$ and $19$. Thus, the only $D^L$-integral dumbbell graphs when $n=6$ are  $DB(W_{7,6})$, $DB(W_{8,6})$, $DB(W_{12,6})$ and $DB(W_{19,6})$.
\end{proof}

\vspace{1cm}

\noindent{\bf Acknowledgements.} The research of S. Pirzada is supported by NBHM-DAE research project No. NBHM/02011/20/2022. The research of Leonardo de Lima is partially supported by CNPq grant 305988/2025-5.\\

\noindent{\bf Data availability} Data sharing does not apply to this article as no data sets were generated or analyzed during the current study.\\

\noindent{\bf Conflict of interest.} The authors declare that they have no conflict of interest.

\end{document}